\documentclass[aps,pre,preprint,floatfix]{revtex4}

\usepackage{latexsym}
\usepackage[usenames]{color}
\usepackage{amsthm}
\usepackage{hyperref}

\usepackage{amsmath}    
\usepackage{amsfonts}  
\usepackage{amssymb}
\usepackage{graphicx}
\usepackage{slashed}
\usepackage{fouridx}

\theoremstyle{plain} 
\theoremstyle{plain} 
\theoremstyle{plain} \newtheorem{theor}{Theorem}[section] 
\theoremstyle{plain} 
\theoremstyle{plain} \newtheorem*{corol}{Corollary} 
\theoremstyle{remark} 
\theoremstyle{plain} 
\theoremstyle{remark}

\newcommand\CROSS[1]{%
  \hbox{%
    \vbox{
      \hrule
      \kern2.5pt
      \hbox{$#1$\,\strut}
    }%
  \vrule
  }\mskip\thickmuskip
}

\begin{document}
\begin{center}
\Large{\textbf{On the characteristic function of a collection of sets}}\\ 
~\\

\large{Vladimir Garc\'{\i}a-Morales}\\

\normalsize{}
~\\

Departament de Termodin\`amica, Universitat de Val\`encia, \\ E-46100 Burjassot, Spain
\\ garmovla@uv.es
\end{center}
\small{The union of a collection of $n$ sets is generally expressed in terms of a characteristic (indicator) function that contains $2^{n}-1$ terms. In this article, a much simpler expression is found that requires the evaluation of $n$ terms only. This leads to a major simplification of any normal form involving characteristic functions of sets. The formula can be useful in recognizing inclusion-exclusion patterns of combinatorial problems.}
\noindent  
~\\

\pagebreak


\section{Introduction}
The characteristic function is a useful means to define a set  \cite{Whitney1, Brown, Kung, Galambos} and allows the formulas of the algebra of logic to be expressed as equations of ordinary algebra \cite{Whitney1}. Let $U$ denote the universe of discourse and let $x\in U$ be one of its elements. Let $A \subseteq U$ be a set. The characteristic function of $A$, $\text{Set}(x; A): U\to \{0,1\}$, is defined as
\begin{align}\label{charfu}
\begin{split}
\text{Set}\left(x; A\right)={\begin{cases} \text1&{\text{if }}x \in A \\ 0 &{\text{if }}x \notin A  \end{cases}}
\end{split}
\end{align}
Obviously, $\text{Set}\left(x; U\right)=1$. Let $B \subseteq U$ be another set. The characteristic function $\text{Set}\left(x; A\cap B \right)$ of the intersection $A\cap B$ is the product of the individual characteristic functions of the sets $A$ and $B$
\begin{align}\label{charfu}
\begin{split}
\text{Set}\left(x; A\cap B \right)=\text{Set}\left(x; A\right)\text{Set}\left(x; B\right)
\end{split}
\end{align}
If $A$ and $B$ are \emph{disjoint}, their intersection is equal to the empty set $A\cap B=\emptyset$ and the characteristic function of the union $A \cup B$ of $A$ and $B$ is given by
\begin{align}\label{disjo}
\begin{split}
\text{Set}\left(x; A\cup B  \right)=\text{Set}\left(x; A\right)+\text{Set}\left(x; B\right) 
\end{split}
\end{align}
If $A$ and $B$ are not disjoint, Eq. (\ref{disjo}) counts twice the elements in $A\cap B$ and it is necessary to correct this overcount by subtracting once the contribution of the intersection
\begin{align}\label{uni2}
\begin{split}
\text{Set}\left(x; A\cup B  \right)=\text{Set}\left(x; A\right)+\text{Set}\left(x; B\right)-\text{Set}\left(x; A\cap B  \right) 
\end{split}
\end{align}
If one considers the union of three sets, we have, consequently
\begin{align}\label{uni3}
\begin{split}
\text{Set}\left(x; A\cup B\cup C  \right)&=\text{Set}\left(x; A\right)+\text{Set}\left(x; B\right)+\text{Set}\left(x; C\right)\\
&-\text{Set}\left(x; A\cap B  \right)-\text{Set}\left(x; A\cap C  \right)-\text{Set}\left(x; B\cap C  \right) \\
&+\text{Set}\left(x; A\cap B\cap C  \right) 
\end{split}
\end{align}

An induction argument, employing Eq. (\ref{uni2}) iteratively, shows that the characteristic function $\text{Set}\left(x; \bigcup_{j=1}^n S^{(j)}\right)$ of the union of $n$ sets $S^{(1)}$, $\ldots$, $S^{(n)}$ is given by
\begin{equation}
 \text{Set}\left(x; \bigcup_{j=1}^n S^{(j)}\right)  = \sum_{k = 1}^{n} (-1)^{k+1} \sum_{1 \leq j_{1} < \cdots < j_{k} \leq n} \text{Set}\left(x; \bigcap_{m=1}^k S^{(j_{m})}\right)  \label{inextrape}
\end{equation}
where 
\begin{equation}
\text{Set}\left(x; \bigcap_{m=1}^k S^{(j_{m})}\right)=\prod_{m=1}^{k}\text{Set}\left(x; S^{(j_{m})}\right) \label{interse}
\end{equation}
is the characteristic function of the intersection of $k$ sets. Eq. (\ref{inextrape}) corresponds to \emph{Whitney's first normal form} \cite{Whitney1} for the characteristic function of the union of $n$ sets. 
Eq. (\ref{inextrape}) is closely related to the principle of inclusion-exclusion, which is of major importance in enumerative combinatorics \cite{Rota, Stanley, Aigner2} (see \cite{BonaEC} for an introductory treatment), probability theory \cite{Galambos}, computation \cite{Bjorklund} and graph theory \cite{Whitney2}. Indeed, the inclusion-exclusion principle is swiftly derived from Eq. (\ref{inextrape}) as follows. We note that, from Eq. (\ref{inextrape}), the set of elements $x$ that do \emph{not} belong to the union of the $m$ sets has characteristic function given by
\begin{eqnarray}
 \text{Set}\left(x; U / \bigcup_{j=1}^n S^{(j)}\right) & =& \text{Set}\left(x; U\right)-\text{Set}\left(x; \bigcup_{j=1}^n S^{(j)}\right) \nonumber \\
 &=& 1-\sum_{k = 1}^{n} (-1)^{k+1} \sum_{1 \leq j_{1} < \cdots < j_{k} \leq n} \text{Set}\left(x; \bigcap_{m=1}^k S^{(j_{m})}\right)  \nonumber \\ 
&=& \sum_{k = 0}^{n} (-1)^{k} \sum_{1 \leq j_{1} < \cdots < j_{k} \leq n} \text{Set}\left(x; \bigcap_{m=1}^k S^{(j_{m})}\right)  \label{inextrape2}
\end{eqnarray}
Let us define the cardinal (size) $\text{Card}(A)$ of a finite set $A$ as
\begin{equation}
\text{Card}(A)=\sum_{x\in U} \text{Set}\left(x; A\right) \label{cardinal}
\end{equation}
If we assume that all sets entering in Eq. (\ref{inextrape2}) are finite, we then have, from Eqs. (\ref{inextrape2}) and Eq. (\ref{cardinal})
\begin{eqnarray}
 \text{Card}\left(U / \bigcup_{j=1}^n S^{(j)}\right) & =& \sum_{k = 0}^{n} (-1)^{k} \sum_{1 \leq j_{1} < \cdots < j_{k} \leq n} \text{Card}\left(\bigcap_{m=1}^k S^{(j_{m})}\right)  \label{inextrape3}
\end{eqnarray}
which is the celebrated inclusion-exclusion principle \cite{Aigner2}. Eq. (\ref{inextrape3}) gives the size of the set whose elements belong to \emph{none} of the sets $S^{(j)}$, $j=1, \ldots, m$. 

We note the following striking fact. While the intersection of $n$ sets only requires the product of their characteristic functions (i.e. the product of $n$ terms), the high degree of complexity of Eq. (\ref{inextrape}) is remarkable, since it contains a total of $2^{n}-1$ terms. This is easily checked for Eqs. (\ref{uni2}) and (\ref{uni3}) above (both specific instances of Eq. (\ref{inextrape})) whose right hand sides contain, respectively, $2^{2}-1=3$ and $2^{3}-1=7$ terms.  The complexity of this formula is caused by the term $-\text{Set}\left(x; A\cap B  \right)$ in Eq. (\ref{uni2}), which makes the pairwise union of two sets to depart from the mere ordinary addition of their characteristic functions. Since, as a consequence of the postulates of set theory, any arbitrary non-empty set can be decomposed into a combination of unions and intersections of sets of equal or lower size, the question whether a simpler formula for the characteristic function of the union of $n$ sets exists is an interesting one and may find applications in all branches of mathematics. Usually, the discovery of inclusion-exclusion patterns is associated to major advances towards the solution of subtle combinatorial problems. It is to be noted, however, that it is generally very hard to recognize them. This difficulty is associated to the complexity of the inclusion-exclusion principle, that is directly inherited from the complexity of Eq. (\ref{inextrape}), as described above. For example, it took fifty-five years \cite{Rota}, since Cayley's former attempts, before Jacques Touchard in 1934 could discover an inclusion-exclusion pattern for the m\'enage problem and, thence, readily obtain the solution to the problem as an explicit binomial formula \cite{Rota}. 

In this article, we establish a normal form for the characteristic function of the union of $n$ sets, see Eq. (\ref{mainU}) below, that has a simple geometric interpretation and is simpler than Eq. (\ref{inextrape}) in the sense that \emph{it only requires the evaluation of $n$ terms, i.e. the individual characteristic functions $\text{\emph{Set}}\left(x; S^{(j)}\right)$, $1\le j \le n$ of each of the $n$ sets}. This \emph{polynomial} dependence on $n$, the number of sets in the collection, strikingly contrasts with the \emph{exponential} dependence on $2^{n}-1$ terms of Whitney's normal form, Eq. (\ref{inextrape}). Our expression considerably simplifies the description of the union of $n$ arbitrary sets (and, of course, the addition of further sets to the union). Once understood, it is trivial to write down any specific form of our expression in any mathematical problem where the union of certain sets is involved, and we believe that it may be of help to combinatorialists in the discovery of inclusion-exclusion patterns. 

The outline of this article is as follows. In Section \ref{mathprembfun} we concisely present the main mathematical structure, the $\mathcal{B}$-function, that is required to derive and understand our expression. We also briefly sketch how this function can be used as a building block to construct characteristic functions of arbitrary sets (which can be numerable or not).  In Section \ref{core} we state our main result as a theorem and prove it, together with several corollaries. The theorem provides a toolbox for the analysis of arbitrarily complex sets and a means to join or decompose them into parts, no matter how intricate the intersections can be. In Section \ref{discu} we discuss the theorem, giving examples of it, as well as its implications and its geometric underpinnings.

\section{Mathematical prerequisite: The $\mathcal{B}$-function} \label{mathprembfun}

In previous works we have introduced a $\mathcal{B}$-function in the discussion of rule-based dynamical systems, as cellular automata \cite{VGM1,VGM2,VGM3,JPHYSA} and substitution systems \cite{VGM4}. Since the $\mathcal{B}$-function is crucial in the next section (in both stating and proving our main result) we briefly discuss it here. For arbitrary $x, y \in \mathbb{R}$ the $\mathcal{B}$-function is defined by the relationship
\begin{equation}
\mathcal{B}(x,y)=\frac{1}{2}\left(\frac{x+y}{|x+y|}-\frac{x-y}{|x-y|}\right)=\frac{1}{2}\left(\text{sign}(x+y)-\text{sign}(x-y)\right)={\begin{cases} \text{sign}(y)&{\text{if }}|x| < |y|\\ \text{sign}(y)/2 &{\text{if }}|x|=|y|, y\ne 0 \\0& {\text{otherwise}} \end{cases}} \label{d1}
\end{equation}
and has the form of a rectangular function \cite{JPHYSA} whose thickness is controlled by the value of $y$ and whose height can be positive or negative, depending on the sign of $y$. 

If $x\in \mathbb{Z}$ is constrained to be an integer and $y \in \left.\left(0,\frac{1}{2}\right.\right]$, $y\in \mathbb{R}$, the $\mathcal{B}$-function, Eq. (\ref{d1}), becomes the Kronecker delta function of its first argument
\begin{equation}
\mathcal{B}(x,y)=\delta_{x0}={\begin{cases} 1&{\text{if }}x =0, \qquad \qquad \quad \quad  x\in \mathbb{Z}, y \in \left.\left(0,\frac{1}{2}\right.\right], y \in \mathbb{R} \\ 0 &{\text{if }}x\ne 0 \qquad \qquad \quad \quad \ \ x\in \mathbb{Z}, y \in \left.\left(0,\frac{1}{2}\right.\right], y \in \mathbb{R} \end{cases}} \label{d1Kro}
\end{equation}
and, therefore, if $m, n \in \mathbb{Z}$, then $\mathcal{B}\left(m-n,\varepsilon \right)=\mathcal{B}\left(m-n,\frac{1}{2}\right)=\delta_{mn}$ for any real $\varepsilon$ such that $\varepsilon \in \left.\left(0,\frac{1}{2}\right.\right]$. Also clear is the fact that, if both $x$ and $y$ are non-negative integers, we have
\begin{equation}
\mathcal{B}(x,y)={\begin{cases} 1&{\text{if }}x < y, \qquad \qquad \quad \quad \ \ x,y \in \mathbb{N}\cup\{0\} \\ 1/2 &{\text{if }}x=y\ne 0 \qquad \qquad \quad \ x,y \in \mathbb{N}\cup\{0\} \\0&{\text{if }}x>y  \ {\text{or }}x=y=0, \quad x,y\in \mathbb{N}\cup\{0\} \end{cases}} \label{d1b}
\end{equation}

The following identity, called the splitting property, for any $x, y, z \in \mathbb{R}$
\begin{equation}
\mathcal{B}(x,y+z)=\mathcal{B}(x+y,z)+\mathcal{B}(x-z,y) \label{splito} 
\end{equation}
is simple to prove
\begin{eqnarray}
&&\mathcal{B}(x+y,z)+\mathcal{B}(x-z,y)=\frac{1}{2}\left(\text{sign}(x+y+z)-\text{sign}(x+y-z)\right)+\nonumber \\
&&+\frac{1}{2}\left(\text{sign}(x-z+y)-\text{sign}(x-z-y)\right)=\frac{1}{2}\left(\text{sign}(x+y+z)-\text{sign}(x-y-z)\right) \nonumber \\
&&=\mathcal{B}(x,y+z) \nonumber
 \nonumber \qedhere
\end{eqnarray}

The $\mathcal{B}$-function allows any \emph{inequality} of the form $r<x<q$ involving the real numbers $r$ and $q$ ($r<q$) to be replaced by an \emph{identity}
\begin{equation}
\mathcal{B}\left(x-\frac{q+r}{2},\frac{q-r}{2}\right)=1 \qquad \text{(i.e)} \qquad 
\frac{1}{2}\left(\frac{x-r}{|x-r|}-\frac{x-q}{|x-q|}\right) =1 \label{interesting1}
\end{equation}

This observation can be further generalized to other kinds of inequalities. We note that, by definition, the $\mathcal{B}$-function, as defined above, returns $\pm 1/2$ at the borders, the sign depending on the second argument. If we are interested that the $\mathcal{B}$-function is equal to $\pm 1$ or $0$ at the borders, we can, instead, use the functions $\mathcal{B}_{++}(x,y)$, $\mathcal{B}_{--}(x,y)$, $\mathcal{B}_{+-}(x,y)$ and $\mathcal{B}_{-+}(x,y)$ defined as follows
\begin{eqnarray}
\mathcal{B}_{++}(x,y)&\equiv& \mathcal{B}(0,\mathcal{B}(x,y))={\begin{cases} 1&{\text{if }}x\in [-y,y], y>0\\-1&{\text{if }}x\in [y,-y], y<0\\0&{\text{if }}x\notin [-|y|,|y|] \end{cases}} \label{bclos} \\
\mathcal{B}_{--}(x,y)&\equiv& 2\mathcal{B}(x,y)-\mathcal{B}_{++}(x,y)={\begin{cases} 1&{\text{if }}x\in (-y,y), y>0\\-1&{\text{if }}x\in (y,-y), y<0\\0&{\text{if }}x\notin (-|y|,|y|) \end{cases}} \label{bope} \\
\mathcal{B}_{-+}(x,y)&\equiv& \mathcal{B}\left(0,|x+y|\mathcal{B}(x,y)\right)={\begin{cases} 1&{\text{if }}x\in (-y,y], y>0\\-1&{\text{if }}x\in [y,-y), y<0\\0&{\text{if }}x\notin (-|y|,|y|] \end{cases}} \label{bminplus} \\
\mathcal{B}_{+-}(x,y)&\equiv& \mathcal{B}\left(0,|x-y|\mathcal{B}(x,y)\right)={\begin{cases} 1&{\text{if }}x\in [-y,y), y>0\\-1&{\text{if }}x\in (y,-y], y<0\\0&{\text{if }}x\notin [-|y|,|y|) \end{cases}} \label{bplusmin}
 \end{eqnarray} 
Thus, for example, the characteristic function $\text{Set}(x; [a,b])$ of the closed interval $[a,b]$ ($a<b$) is
\begin{equation}
\text{Set}(x; [a,b])=\mathcal{B}_{++}\left(x-\frac{b+a}{2},\frac{b-a}{2} \right)
\end{equation}
As another example, the characteristic function $\text{Set}(x,y; x^{2}+y^{2}<R)$ of an open disk containing all those points $(x,y)$ in the plane $\mathbb{R}^{2}$ for which $x^{2}+y^{2}<R^{2}$ is
\begin{equation}
\text{Set}(x,y;\ x^{2}+y^{2}<R^{2})=\mathcal{B}_{--}\left(x^{2}+y^{2}, R^{2} \right) \label{opendisk}
\end{equation}

\section{Characteristic functions of collections of sets} \label{core}

We now prove the following theorem which, together with its consequences, constitutes the main result of this article.

\begin{theor} \label{Utheo} Let $U$ denote the universe of discourse, let $x\in U$ be an element of $U$, and let a collection of $n$ sets included in $U$ be denoted by $S^{(1)}$,$\ldots$, $S^{(n)}$. The characteristic functions $\text{\emph{Set}}\left(x; S_{= m}\right)$, $\text{\emph{Set}}\left(x; S_{\le m}\right)$ and $\text{\emph{Set}}\left(x; S_{> m}\right)$ of the elements $x$ that \emph{exactly} belong to $m$ sets, to \emph{no more than $m$} sets and to \emph{more than} $m$ sets of the collection are, respectively, given by 
\begin{eqnarray}
\text{\emph{Set}}\left(x; S_{= m}\right)&=&\mathcal{B}\left(m-\sum_{j=1}^{n}\text{\emph{Set}}\left(x; S^{(j)}\right), \varepsilon \right)  \qquad \forall \varepsilon \in \left.\left(0,\frac{1}{2}\right.\right], \varepsilon \in \mathbb{R} \label{subsets1} \\
\text{\emph{Set}}\left(x; S_{\le m}\right)&=&\mathcal{B}\left(m+1-2\sum_{j=1}^{n}\text{\emph{Set}}\left(x; S^{(j)}\right), m \right)  \label{subsets2} \\
\text{\emph{Set}}\left(x; S_{> m}\right)&=&\mathcal{B}\left(n+m+1-2\sum_{j=1}^{n}\text{\emph{Set}}\left(x; S^{(j)}\right), n-m\right) \label{subsets3}
\end{eqnarray}
\end{theor}

\begin{proof}
From Eq. (\ref{d1}), we have that $\mathcal{B}\left(m-\sum_{j=1}^{n}\text{Set}\left(x; S^{(j)}\right), \varepsilon \right)=1$, for any $\varepsilon \in \left.\left(0,\frac{1}{2}\right.\right], \varepsilon \in \mathbb{R}$, only if
\begin{equation}
-\varepsilon<m-\sum_{j=1}^{n}\text{Set}\left(x; S^{(j)}\right)<\varepsilon \nonumber
\end{equation}
We note that since both $m$ and $\sum_{j=1}^{n}\text{Set}\left(x; S^{(j)}\right)$ are non-negative integers, the only possibility to satisfy this inequality is
\begin{equation}
\sum_{j=1}^{n}\text{Set}\left(x; S^{(j)}\right)=m \nonumber
\end{equation}
which implies that $x$ exactly belongs to $m$ sets. Otherwise, $\mathcal{B}\left(m-\sum_{j=1}^{n}\text{Set}\left(x; S^{(j)}\right), \frac{1}{2} \right)=0$. The same conclusions are trivially reached from Eq. (\ref{d1Kro}) since the $\mathcal{B}$-function reduces to a Kronecker delta of the first argument in this case. Therefore, Eq. (\ref{subsets1}) follows.

Eq. (\ref{subsets2}) is proved in a similar way. We have $\mathcal{B}\left(m+1-2\sum_{j=1}^{n}\text{Set}\left(x; S^{(j)}\right), m \right)=1$ if
\begin{equation}
-m<m+1-2\sum_{j=1}^{n}\text{Set}\left(x; S^{(j)}\right)<m \nonumber
\end{equation}
and zero otherwise. Let $\sum_{j=1}^{n}\text{Set}\left(x; S^{(j)}\right)=k$. The inequality is then satisfied if 
\begin{equation}
0<1+2(m-k)<2m \nonumber
\end{equation}
i.e. if $1 \le k \le m$. This means that $x$ belongs to, at least, one set of the collection, and to no more than $m$ sets. If $k=0$ we have $\mathcal{B}\left(m+1-2\sum_{j=1}^{n}\text{Set}\left(x; S^{(j)}\right), m \right)=\mathcal{B}\left(m+1, m \right)=0$ and if $k>m$, say $k=m+D$ with $D\ge 1$, we obtain $\mathcal{B}\left(m+1-2k, m \right)=\mathcal{B}\left(-m-2D+1, m \right)=\mathcal{B}\left(m+2D-1, m \right)=0$ (since $m+2D-1>m$). Therefore, the characteristic function $\text{Set}\left(x; S_{\le m}\right)$ is given by Eq. (\ref{subsets2}).

Finally, to prove Eq. (\ref{subsets3}) note that, if $\mathcal{B}\left(n+m+1-2\sum_{j=1}^{n}\text{Set}\left(x; S^{(j)}\right), n-m\right)=1$, then
\begin{equation}
-n+m<  n+m+1-2\sum_{j=1}^{n}\text{Set}\left(x; S^{(j)}\right)     <n-m  \nonumber
\end{equation}
Let $\sum_{j=1}^{n}\text{Set}\left(x; S^{(j)}\right)=k$. Thus,
\begin{equation}
0<  2(n-k)+1    <2(n-m) \qquad \to \qquad -2(n-m)<2m+1<2k \qquad \to \qquad m<k  \nonumber
\end{equation}
This means that $\sum_{j=1}^{n}\text{Set}\left(x; S^{(j)}\right)>m$, i.e. that $x$ belongs to more than $m$ sets of the collection. If this is not the case, $\mathcal{B}\left(n+m+1-2\sum_{j=1}^{n}\text{Set}\left(x; S^{(j)}\right), n-m\right)=0$, since then $\sum_{j=1}^{n}\text{Set}\left(x; S^{(j)}\right)=k<m$ implies $n+m+1-2\sum_{j=1}^{n}\text{Set}\left(x; S^{(j)}\right)=n+m+1-2k>n+m+1-2m=n-m+1>n-m$. Therefore, the characteristic function $\text{Set}\left(x; S_{> m}\right)$ is given by Eq. (\ref{subsets3}), as claimed.
\end{proof}

\begin{corol} \label{inex} The characteristic function $ \text{\emph{Set}}\left(x; \bigcup_{j=1}^n S^{(j)}\right)$ of the union of $n$ sets $S^{(1)}$,$\ldots$, $S^{(n)}$ is
\begin{eqnarray}
 \text{\emph{Set}}\left(x; \bigcup_{j=1}^n S^{(j)}\right)  &=& \mathcal{B}\left(n+1-2\sum_{j=1}^{n}\text{\emph{Set}}\left(x; S^{(j)}\right),\ n\right)  \label{mainU} \\
 &=&\frac{1}{2}\left(\frac{2n+1-2\sum_{j=1}^{n}\text{\emph{Set}}\left(x; S^{(j)}\right)}{|2n+1-2\sum_{j=1}^{n}\text{\emph{Set}}\left(x; S^{(j)}\right)|}-\frac{1-2\sum_{j=1}^{n}\text{\emph{Set}}\left(x; S^{(j)}\right)}{|1-2\sum_{j=1}^{n}\text{\emph{Set}}\left(x; S^{(j)}\right)|}\right) \nonumber
\end{eqnarray}
\end{corol}

\begin{proof} The result follows by taking $m=n$ in Eq. (\ref{subsets2}) or $m=0$ in Eq. (\ref{subsets3}). The second of the identities just results from the definition of the $\mathcal{B}$-function, Eq. (\ref{d1}).
\end{proof}

\begin{corol} \label{inex2} 
\begin{equation}
 \text{\emph{Set}}\left(x; S^{(j)}\right)  = \mathcal{B}\left(2-2\text{\emph{Set}}\left(x; S^{(j)}\right),\ 1\right) \label{mainUdos}
\end{equation}
\end{corol}

\begin{proof} Take $n=1$ in Eq. (\ref{mainU}) and note that $S^{(1)}$ may indeed refer to any set $S^{(j)}$ in the union since the $n$ different values of the integer label $j$ can be freely attributed. 
\end{proof}


It is also easy to prove that
 \begin{equation}
\text{Set}\left(x; \bigcup_{j=1}^n S^{(j)}\right)=\text{Set}\left(x; S_{\le m}\right)+\text{Set}\left(x; S_{> m}\right)=\sum_{m=1}^{n}\text{Set}\left(x; S_{=m}\right)
\end{equation}
since, from Eq. (\ref{mainU}) and by the splitting property, Eq. (\ref{splito})
 \begin{eqnarray}
&&\text{Set}\left(x; \bigcup_{j=1}^n S^{(j)}\right)=\mathcal{B}\left(n+1-2\sum_{j=1}^{n}\text{Set}\left(x; S^{(j)}\right),n \right) \nonumber \\
&&=\mathcal{B}\left(m+1-2\sum_{j=1}^{n}\text{Set}\left(x; S^{(j)}\right), m\right) + \mathcal{B}\left(n+m+1-2\sum_{j=1}^{n}\text{Set}\left(x; S^{(j)}\right),n-m\right)
 \nonumber \\ 
 &&= \text{Set}\left(x; S_{\le m}\right)+\text{Set}\left(x; S_{> m}\right) \nonumber 
\end{eqnarray}
The last identity $\text{Set}\left(x; S_{\le m}\right)+\text{Set}\left(x; S_{> m}\right)=\sum_{m=1}^{n}\text{Set}\left(x; S_{=m}\right)$ is proved by induction. For $n=1$ it is trivially true. Let us assume it valid for a collection of $n-1$ sets. Then, for $n$ sets we have $\text{Set}\left(x; \bigcup_{j=1}^n S^{(j)}\right)=\text{Set}\left(x; S_{\le n-1}\right)+\text{Set}\left(x; S_{> n-1}\right)=\text{Set}\left(x; S_{\le n-1}\right)+\text{Set}\left(x; S_{=n}\right)=\sum_{m=1}^{n-1}\text{Set}\left(x; S_{=m}\right)+\text{Set}\left(x; S_{=n}\right)=\sum_{m=1}^{n}\text{Set}\left(x; S_{=m}\right)$.

\section{Discussion} \label{discu}

The simplicity of our expression for the characteristic function Eq. (\ref{mainU}), compared to Eq. (\ref{inextrape}), is apparent: It only makes use of the characteristic functions of the individual sets $\text{Set}\left(x; S^{(j)}\right)$ and it thus depends only on $n$ terms. Furthermore, although our characteristic function is defined through the $\mathcal{B}$-function that is a nonlinear function, the two arguments of the latter change linearly as more sets are being added to the union. For each set $A$ that joins a previously formed union of sets, it is only necessary to add $1-2\text{Set}\left(x; A\right)$, to the first argument of the $\mathcal{B}$-function and `1' to the second one. Let us illustrate this process by starting with the empty set that can be represented with the characteristic function
\begin{equation}
\text{Set}\left(x; \emptyset \right)=\mathcal{B}(1,0)=0 
\end{equation}
if we now consider only one set $S^{(1)}$ and, hence, $n=1$ then, by adding `$1-2\text{Set}\left(x; S^{(1)}\right)$' to the first argument of the $\mathcal{B}$-function and `1' to the second argument we obtain
\begin{equation}
 \text{Set}\left(x; S^{(1)}\right)  = \mathcal{B}\left(1+1-2\text{Set}\left(x; S^{(1)}\right),\ 0+1\right)=\mathcal{B}\left(2-2\text{Set}\left(x; S^{(1)}\right),\ 1\right) \label{mainUdos2}
\end{equation} 
which is Eq. (\ref{mainUdos}). If another set $S^{(2)}$ is joined we just add $1-2\text{Set}\left(x; S^{(2)}\right)$ to the first argument and `1' to the second argument of the $\mathcal{B}$-function in Eq. (\ref{mainUdos2}) to obtain
\begin{eqnarray}
 \text{Set}\left(x; S^{(1)}\cup S^{(2)} \right) &=&\mathcal{B}\left(2-2\text{Set}\left(x; S^{(1)}\right)+1-2\text{Set}\left(x; S^{(2)}\right),\ 1+1\right) \nonumber \\
 &=&\mathcal{B}\left(3-2\left[\text{Set}\left(x; S^{(1)}\right)+\text{Set}\left(x; S^{(2)}\right)\right],\ 2\right)
\end{eqnarray} 
This is Eq. (\ref{mainU}) with $n=2$. This process can go on indefinitely. Note how its complexity increases linearly with $n$ compared to the complexity of Eq. (\ref{inextrape}) that increases as $2^{n}-1$.

\begin{figure*}[!h]
\centering\includegraphics[width=3.5in,angle=270]{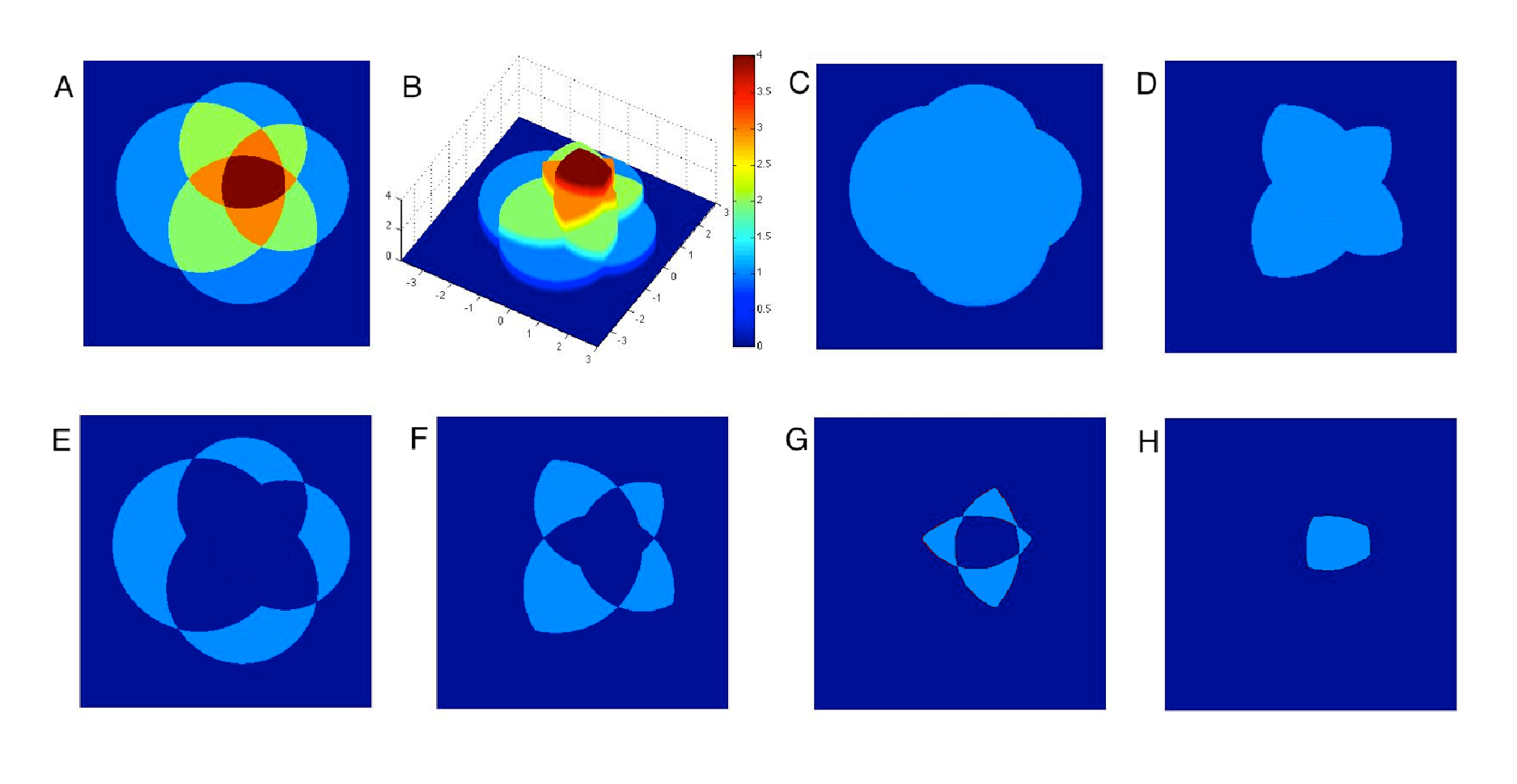}
\caption{Plot of the function $\sum_{j=1}^{4}\text{Set}\left(z; S^{(j)}\right)$ for $\text{Set}\left(z; S^{(j)}\right)$, $j=1,\ldots 4$ given by Eqs. (\ref{S1}) to (\ref{S4}) in the plane (A) and in three dimensions (B). For the same collection of sets the following characteristic functions are plotted: $ \text{ Set}\left(z; \bigcup_{j=1}^4 S^{(j)}\right)$ (C), calculated from Eq. (\ref{mainU}); $ \text{ Set}\left(x; S_{> 1}\right)$ (D), calculated from Eq. (\ref{subsets3}), and $ \text{Set}\left(x; S_{=m}\right)$, calculated from Eq. (\ref{subsets1}) for $m=1$ (E), $m=2$ (F), $m=3$ (G) and $m=4$ (H). Shown is, in every case, the square region of the plane comprised between the points $(-3.75 ,-3.75  )$ and $( 3, 3)$.}
\label{setas}
\end{figure*}

Our characteristic function also admits a straightforward geometric interpretation as illustrated with the following example. Let us, e.g. consider as the universe $U$ the Euclidean plane $\mathbb{R}^{2}$. Thus, an element $z$ of $U$ is any pair of coordinates $z=(x,y)$ that specify a point in the plane. We consider the following sets
\begin{eqnarray}
\text{Set}(z; S^{(1)})&\equiv &\text{Set}\left(x,y;\ (x-1)^{2}+y^{2}<1.5^{2}\right)=\mathcal{B}_{--}\left((x-1)^{2}+y^{2}, 2.25 \right) \label{S1} \\
\text{Set}(z; S^{(2)})&\equiv &\text{Set}\left(x,y;\ (x+1)^{2}+y^{2}<2^{2}\right)=\mathcal{B}_{--}\left((x+1)^{2}+y^{2}, 4 \right) \label{S2} \\
\text{Set}(z; S^{(3)})&\equiv &\text{Set}\left(x,y;\ x^{2}+(y-1)^{2}<1.5^{2}\right)=\mathcal{B}_{--}\left(x^{2}+(y-1)^{2}, 2.25 \right) \label{S3} \\
\text{Set}(z; S^{(4)})&\equiv &\text{Set}\left(x,y;\ x^{2}+(y+1)^{2}<1.75^{2}\right)=\mathcal{B}_{--}\left(x^{2}+(y+1)^{2}, 3.0625 \right) \label{S4} 
\end{eqnarray} 
i.e.  a collection of four open disks, as defined in Eq. (\ref{opendisk}), in the Euclidean plane. We can, from these definitions explicitly calculate all characteristic functions in Theorem \ref{Utheo} for points $(x,y)$ in the plane, and plot them in 3D as surfaces. It is first useful to consider the function $\sum_{j=1}^{4}\text{Set}\left(z; \bigcup_{j=1}^n S^{(j)}\right)$. Since there are 4 sets, this function is integer valued and takes values on the interval $[0,4]$. The larger the value of $\sum_{j=1}^{4}\text{Set}\left(z; \bigcup_{j=1}^n S^{(j)}\right)$, the larger the number of sets to which $z$ does belong. In Fig. \ref{setas}A, we plot this function in the plane. We observe that the four disks have regions of intersection and, as a result, the function resembles a colored Venn diagram in the plane, the colors ranging from deep blue to deep red as $\sum_{j=1}^{4}\text{Set}\left(z; \bigcup_{j=1}^n S^{(j)}\right)$ ranges from $0$ to $4$. In Fig. \ref{setas}B we plot this function in 3D and it becomes clear that those points that belong to a larger number of sets have a larger height over the plane. The function resembles a `tower' and the $\mathcal{B}$-function can then be used at different levels to cut this tower into pieces of unit height. The resulting pieces are sets whose characteristic functions are, precisely, the ones obtained from Theorem \ref{Utheo}. The union of the four sets above $ \text{ Set}\left(z; \bigcup_{j=1}^4 S^{(j)}\right)$ as obtained from Eq. (\ref{mainU}) is shown in Fig. \ref{setas}C. In the rest of the panels of the figure, the following characteristic functions are shown: $ \text{ Set}\left(x; S_{> 1}\right)$ (D), calculated from Eq. (\ref{subsets3}), and $ \text{Set}\left(x; S_{=m}\right)$, calculated from Eq. (\ref{subsets1}) for $m=1$ (E), $m=2$ (F), $m=3$ (G) and $m=4$ (H). We see that, in every case, the $\mathcal{B}$-function allows to extract pieces from the range of integers $\in [0,4]$ of the mere superposition of the characteristic functions $ \text{ Set}\left(z; \bigcup_{j=1}^4 S^{(j)}\right)$, projecting the result to the set $\{0,1\}$ and allowing the result to be interpreted in terms of the membership $\in$ relation. The $\mathcal{B}$-function thus produces a `tomography' of any linear superposition of characteristic functions of sets.

As another example of the application of Theorem \ref{Utheo}, let us consider as universe $U$ the natural numbers, and the set of those natural numbers $j\in \mathbb{N}$ that divide $x\in U$, i.e. those $j$ such that $j|x$. Its characteristic function is given by
\begin{equation}
\text{Set}(j ; j|x)\equiv \mathcal{B}\left(\mathbf{d}_j(0,x), \frac{1}{2} \right) \label{divide}
\end{equation}
where the zeroth digit function \cite{CHAOSOLFRAC,PHYSAFRAC}
\begin{equation}
\mathbf{d}_{j}(0,x)\equiv x-j\left \lfloor \frac{x}{j} \right \rfloor \nonumber
\end{equation}
returns the remainder of dividing $x$ by $j$. Here the brackets $\left \lfloor \ldots \right \rfloor$ denote the integer part (floor) function. Clearly, the $\mathcal{B}$-function  in Eq. (\ref{divide}) returns 1 if the remainder is zero and 0 otherwise. 

Thus, since a number $j > \left \lfloor x/2 \right \rfloor$ cannot be a proper divisor of $x$, the total number of proper divisors of $x$ is given by
\begin{equation}
\sum_{j=2}^{\left \lfloor x/2 \right \rfloor} \text{Set}(j ; j|x)= \sum_{j=2}^{\left \lfloor x/2 \right \rfloor}\mathcal{B}\left(\mathbf{d}_j(0,x), \frac{1}{2} \right)  \label{prime2}
\end{equation}
From Eq. (\ref{subsets1}) we have that the set of all those natural numbers $x$ with exactly $m$ proper divisors has characteristic function
\begin{equation}
\mathcal{B}\left(m-\sum_{j=2}^{\left \lfloor x/2 \right \rfloor}\mathcal{B}\left(\mathbf{d}_j(0,x), \frac{1}{2} \right), \frac{1}{2} \right)
\end{equation}
and, therefore, the characteristic function for the set of all prime numbers is obtained from this latter expression by taking $m=0$ as
\begin{equation}
\text{Set}\left(x; \text{Prime}\right) = \mathcal{B}\left(\sum_{j=2}^{\left \lfloor x/2 \right \rfloor} \mathcal{B}\left(\mathbf{d}_j(0,x), \frac{1}{2} \right), \frac{1}{2} \right)  \label{bsetprime}
\end{equation}

The total number of primes less or equal than $N\in \mathbb{N}$ is given by
\begin{equation}
\text{Card}\left(\text{Prime}\le N\right)= \sum_{x=2}^{N}\text{Set}\left(x; \text{Prime}\right)=\sum_{x=2}^{N}\mathcal{B}\left(\sum_{j=2}^{\left \lfloor x/2 \right \rfloor} \mathcal{B}\left(\mathbf{d}_j(0,x), \frac{1}{2}\right), \frac{1}{2}\right)  \label{Noprimes}
\end{equation}

\section{Conclusion} \label{conclu}

In this article, formulae have been established for characteristic functions of a collection of $n$ sets which are both simple and powerful (enjoying some nice algebraic properties). This leads to a major simplification of the algebra of logic, from a traditional expression, as Eq. (\ref{inextrape}), that contain up to $2^{n}-1$ terms, to our formula, Eq. (\ref{mainU}), which involves just $n$ terms only. No detailed knowledge of the intersections of the $n$ individual sets in the collection is required to evaluate our formulae: The sets are here handled as if they were disjoint, but their mutual overlaps(non-empty intersections) are appropriately accounted for by means of a $\mathcal{B}$-function which is able to discriminate between different layers of sets (degrees of overlapping). The $\mathcal{B}$-function has nice additive properties: If more sets are added to the collection, this is reflected on a mere linear change of the two arguments on which the $\mathcal{B}$-function depends. This behavior is linked to the splitting property of the $\mathcal{B}$-function Eq. (\ref{splito}). Furthermore, we have suggested how the $\mathcal{B}$-function can itself be regarded as a building block to construct characteristic functions of arbitrarily complex sets. 

In a recent paper \cite{JPHYSA} we have introduced a one-parameter family of functions called $\mathcal{B}_{\kappa}$-functions. The latter become the $\mathcal{B}$-function in the proper limit $\kappa \to 0$ of the real-valued parameter $\kappa$ and, for finite and nonvanishing $\kappa$, they constitute a smooth version of the $\mathcal{B}$-function. It turns out that $\mathcal{B}_{\kappa}$-functions also have the splitting property \cite{JPHYSA} and many other properties of the $\mathcal{B}$-function. Following \cite{JPHYSA}, all expressions found here where a $\mathcal{B}$-function appears, can be approximated by replacing those $\mathcal{B}$-functions by $\mathcal{B}_{\kappa}$-functions, thus smoothing the former and embedding them in the field of the real numbers. Therefore, all expressions coming from our main result, Theorem \ref{Utheo}, can be appropriately smoothed to obtain real-valued, `fuzzy' \cite{Zadeh} counterparts. It is thus straightforward to derive, through $\mathcal{B} \to \mathcal{B}_{\kappa}$ replacements, fuzzy versions of further set-theoretic and logical notions as the ones discussed in this paper. These possibilities, and the implications that these expressions may have in probability theory, shall be investigated elsewhere.

%
%
%
%
%
%
%


\end{document}